\documentclass[letterpaper, 12pt]{article}
\usepackage[left=1in,top=1in,right=1in,bottom=1in]{geometry}

\usepackage[activeacute, english]{babel}
\usepackage[utf8]{inputenc}
\usepackage[T1]{fontenc}

\usepackage{lmodern}
\usepackage{amsmath,amssymb,amsfonts,amsthm}
\usepackage{colonequals}

\usepackage{cite}
\usepackage{cleveref}

\usepackage{enumitem}
\setlist[enumerate]{topsep=1pt,itemsep=-1ex,partopsep=1ex,parsep=1ex}

\theoremstyle{plain}
	\newtheorem{theorem}{Theorem}
	\newtheorem{conjecture}[theorem]{Conjecture}
	\newtheorem{lemma}[theorem]{Lemma}
	\newtheorem{corollary}[theorem]{Corollary}
\theoremstyle{definition}
	\newtheorem{definition}{Definition}

\title{Stability and Ramsey numbers for cycles and wheels}
\author{Nicolás Sanhueza}
\date{}

\begin{document}
\begin{center}
\textbf{Stability and Ramsey numbers for cycles and wheels} \\
	$ $ \\
	Nicolás Sanhueza \\
	\verb|nsanhueza@dim.uchile.cl| \\
	\emph{Departamento de Ingeniería Matemática, Universidad de Chile} \\
	$ $
\end{center}

\begin{abstract}
We study the structure of red-blue edge colorings of complete graphs, with no copies of the $n$-cycle $C_n$ in red, and no copies of the $n$-wheel $W_n = C_n \ast K_1$ in blue, for an odd integer $n$. Our first main result is that in any such coloring, deleting at most two vertices we obtain a vertex-partition of $G$ into three sets such that the edges inside the partition classes are red, and edges between partition classes are blue. As a second result, we obtain bounds for the Ramsey numbers of $r(C_{2k+1},W_{2j})$ for $k < j$ integers, which asymptotically confirm the values of $4j+1$, as it were conjectured by Zhang et al.
\end{abstract}

\section{Introduction}

We study the structure of red-blue edge-colorings in complete graphs, which avoid certain monochromatic subgraphs. More concretely, we consider the case of odd positive integer $n$, and the forbidden monochromatic graphs given by the red $n$-cycle $C_n$ and the blue $n$-wheel $W_n := C_n \ast K_1$. Our main result is the following:

\begin{theorem} \label{theorem:mainresult}
Let $k \ge 6$ and $N \ge 5k+3$. Suppose $G := K_N$ has a red-blue coloring of its edges in a way such that $C_{2k+1}$ is not a red subgraph of $G$ and $W_{2k+1}$ is not a blue subgraph of $G$. Then, there is a partition of $V(G)$ given by $\{ U_0, U_1, U_2, U_3 \}$ such that $|U_0| \leq 2$, $|U_i| \leq 2k$ for $1 \leq i \leq 3$; and every edge in $G - U_0$ inside the partition classes $\{U_1, U_2, U_3\}$ is red, and blue otherwise.

% \[ E(U_i, U_i) \subseteq E^R(G) \quad \text{and} \quad E(U_i, U_j) \subseteq E^B(G), \]for all $i \neq j \in \{1,2,3\}$.
\end{theorem}

A similar result was obtained by Nikiforov and Schelp \cite{nikiforov-schelp}, considering the case where the forbidden monochromatic subgraphs are odd cycles. More precisely, they proved that given $k \ge 2$ and $N \ge 3k+2$, if $G := K_N$ has a red-blue coloring of its edges in a way such that $C_{2k+1}$ is neither a red nor a blue subgraph of $G$, then there is a partition of $V(G)$ given by $\{ U_0, U_1, U_2 \}$ such that $|U_0| \leq 1$ and the edges inside the partition classes $U_1$ and $U_2$ have one color; and are colored with the remaining color otherwise.

%In a way, both our theorem and this last result can be understood as a partial converse of well-known theorems on asymmetric Ramsey numbers. For $G, H$ graphs, the \emph{asymmetric Ramsey number} $r(G,H)$ is the least $n$ such that every red-blue edge-coloring of the complete graph with $n$ vertices contains $G$ as a red subgraph, or $H$ as a blue subgraph. It was observed by Chvátal and Harary \cite{chvatal-harary} says that for all graphs $G$ and $H$, we have that $r(G, H) \ge (\chi(G) - 1)(c(H) - 1) + 1$, where $\chi(G)$ is the chromatic number of $G$ and $c(H)$ is the number of vertices of the largest connected component of $H$. To see this, it suffices to consider the following edge-coloring of the complete graph of order $(\chi(G) - 1)(c(H) - 1)$: partition the vertex set in $(\chi(G) - 1)$ sets, each of size $(c(H)-1)$, and color the edges inside the partition classes blue, and all other edges red. Our result shows that every coloring with forbidden cycles and wheels must necessarily have a similar structure, save for two vertices.

%Our results show that a complete graph colored in a way that it does not contain a red $n$-cycle nor a blue $n$-wheel must \emph{necessarily} have a similar structure to the one used to prove the Chvátal-Harary theorem: save for a fixed finite size set of vertices, the red edges form a $(\chi(G)-1)$-partite graph where each independent component has size less than $c(H)$.

Our proof of \Cref{theorem:mainresult} depends on certain bounds on asymmetric Ramsey numbers. We focus on the case where $G$ is the $n$-cycle $C_n$, and $H$ is the $m$-wheel $W_m$, for $n, m$ integers. Some Ramsey numbers of $C_n$ and $W_m$ are known (\hspace{1sp}\cite{chen-cheng-miao-ng}, \cite{chen-cheng-ng-zhang}, \cite{zhang-zhang-chen}), depending on the parity of $n$ and $m$ and their relative size. In particular, it is known that \[ r(C_n, W_m) = \left\lbrace \begin{array}{ll}
2n - 1 & \text{\emph{for even $m$, with $m \ge 4$, $n \ge 3m/2 - 1$,}} \\
3n - 2 & \text{\emph{for odd $m$, with $n \ge m \ge 3$, $(n,m) \neq (3,3)$,}} \\
2m + 1 & \text{\emph{for odd $n$, with $m \ge 3(n-1)/2$, $(n,m) \neq (3,3), (n,m) \neq (3,4)$,}} \\
3n - 2 & \text{\emph{for odd $n$ and $m$; with $n < m \leq 3(n-1)/2$.}} \\
\end{array} \right. \] Notice that $r(C_n, W_m)$ is not known for odd $n$ and even $m$ with $n<m< 3(n-1)/2$. Zhang et al. \cite{zhang-zhang-chen} raised a conjecture concerning these values.

\begin{conjecture}[Zhang et al. \cite{zhang-zhang-chen}] \label{conjecture:zhang}
Let $n<m$ integers, with $n$ odd and $m$ even. Then $r(C_n, W_m) = 2m+1$.
\end{conjecture} From now on, suppose $n<m$ are integers, with $n=2k+1$ and $m=2j$. We confirm the previous conjecture asymptotically in terms of $j$.

\begin{theorem} \label{theorem:secondaryresult}
Let $2 < k < j$ be integers. We have the following bounds for the Ramsey number of the $(2k+1)$-cycle versus the $2j$-wheel.
\begin{enumerate}[label=\normalfont{(\alph*)}]
\item If $k \ge 3$, then $r(C_{2k+1}, W_{2j}) \leq \frac{9}{2}j+1$.
\item If $k \ge 3$, then $r(C_{2k+1}, W_{2j}) \leq 4j + 334$.
\end{enumerate}
\end{theorem}

In particular, we will make use of the upper bounds of $r(C_{2k+1}, W_{2k+2})$ for our proof of \Cref{theorem:mainresult}. Both bounds of \Cref{theorem:secondaryresult} follow from a more general type of bound that we state and prove in \Cref{section:proofofsecondarytheorem}.

\section{Preliminaries}

We fix a little bit of notation. For every graph $G$, we write $|G|$ and $\Vert G \Vert$ for its number of vertices and edges respectively. The \emph{length} of a path $P$ is $\Vert P \Vert$, its number of edges. For disjoint sets of vertices $A$ and $B$, an \emph{$(A,B)$-path} is a path with one endpoint in $A$, the other in $B$ and no other vertices in $A \cup B$.

% dejar más informal!
Given a red-blue coloring of the edges of a graph $G$, let $G^R$ be the graph on $V(G)$ only containing the red-colored edges, similarly define $G^B$ as the graph on $V(G)$ only containing the blue-colored edges. Let $E^R(G)$ and $E^B(G)$ be the set of edges of $G^R$ and $G^B$, respectively.

%\begin{definition}
%A \emph{$r$-universal graph} is a graph with at least $r$ universal vertices. For every graph, we let $\Upsilon(G)$ be the set of universal vertices of $G$.
%\end{definition}

\begin{definition}
Let $G$ be a graph. A \emph{hedgehog} is a tuple $(W,X)$ where $X \subseteq W \subseteq V(G)$, $X$ induces a complete subgraph and the edges in $E(W \setminus X, X)$ induce a complete bipartite subgraph.
\end{definition}

The notion of hedgehogs will be useful in the proof of \Cref{theorem:mainresult}. The main property of hedgehogs is that every pair of vertices can be joined by paths of various lengths, and that allows us to find cycles of various sizes.

\begin{lemma}\label{lemma:joininghedgehogbylongpaths}
Let $(W, X)$ a hedgehog. Then every pair of distinct vertices in $W$ can be joined by red paths of every length between $2$ and $|X|-1$.
\end{lemma}

\begin{proof}
As $X$ induces a complete subgraph, every pair of distinct vertices in $X$ can be joined by paths of every length between $1$ and $|X|-1$. For distinct pair of vertices in $W$, not necessarily contained in $X$, we can use the edges in $E(W \setminus X, X)$ to extend the mentioned paths or to find a pair of length 2 connecting these vertices, and conclude the result.
\end{proof}

\begin{corollary} \label{corollary:nodisjointrededges}
For every $i \in \{1,2\}$, let $(R_i, S_i)$ be hedgehogs in a graph $G$, such that $R_1 \cap R_2 = \emptyset$. Suppose that $\min \{ |S_1|, |S_2| \} \ge 3$ and that there exist two disjoint edges in $E(R_1, R_2)$. Then $G$ contains cycles of every length between $6$ and $|S_1|+|S_2|$.
\end{corollary}

\begin{proof}
Using \Cref{lemma:joininghedgehogbylongpaths} we can join every two vertices in an hedgehog with paths of various lengths. Choosing these vertices to be the endpoints of two disjoint edges in $E(R_1, R_2)$ we find cycles of the desired lengths.
%Let $e_1, e_2$ be two disjoint edges in $E(R_1, R_2)$, and let $v_i^j := R_i \cap e_j$. By \Cref{lemma:joininghedgehogbylongpaths}, for every $2 \leq j_1 \leq |S_1| - 1$ there exists a $(v_1^1, v_1^2)$-path of length $j_1$ in $S_1$; and for every $2 \leq j_2 \leq |S_2| - 1$ there exists a $(v_2^1, v_2^1)$-path of length $j_2$. Using these paths and the edges $e_1$ and $e_2$, we find red cycles of every length between $6$ and $|S_1| + |S_2|$.
\end{proof}

We shall make use of the values of Ramsey numbers for cycles, which are completely known.

\begin{theorem}[Faudree-Schelp, \cite{faudree-schelp}] \label{lemma:ramseyforcicles}
We have \[ r(C_n, C_m) = \left\lbrace \begin{array}{ll}
6 					& (n,m) \in \{ (3,3), (4,4) \}, \\
2n-1				& \text{$3 \leq m \leq n$, odd $m$, $(n,m) \neq (3,3)$,} \\
n + \frac{m}{2} - 1	& \text{$4 \leq m < n$ both $n, m$ even, $(n,m) \neq (4,4)$,} \\
\max \left\lbrace n + \frac{m}{2} - 1, 2m-1 \right\rbrace & \text{$3 \leq m \leq n$, even $m$ and odd $n$}.
\end{array}  \right. \]
\end{theorem}

\begin{theorem}[Surahmat et al. \cite{surahmat-baskoro-tomescu}] \label{lemma:surahmat} We have that $r(C_{2k+1}, W_{2k+1}) = 6k+1$, for all integers $k \ge 1$.
\end{theorem}

Next, we need some results on the stability of cycle-forbidding red-blue colorings, as shown by Nikiforov and Schelp \cite{nikiforov-schelp}.

\begin{theorem}[Nikiforov-Schelp, \cite{nikiforov-schelp}] \label{lemma:partitionofhamiltoniangraph}
Let $G$ be a hamiltonian graph of order $2n$ such that $C_{2n-1} \nsubseteq G$ and $C_{2n-1} \nsubseteq \overline{G}$. Then there exists a partition of $V(G)$, $\{U_1, U_2\}$ such that $|U_1| = |U_2| = n$ and $U_1, U_2$ are independent. Moreover, there exists a vertex $u \in V(G)$ such that $G - u \cong K_{n, n-1}$.
\end{theorem}

We shall make use of an intermediate result of the same authors, in the same vein.

\begin{lemma}[Nikiforov-Schelp, \cite{nikiforov-schelp}] \label{lemma:prenikiforovschelp}
Let $n = 2k+1 \ge 5$ and $N \ge 3k+2$, and a graph $G := K_N$ with an associated red-blue coloring of its edges $c: E(G) \rightarrow \{R,B\}$ such that $C_{2k+1} \nsubseteq G^R$ and $C_{2k+1} \nsubseteq G^B$. Then there exists a color $C \in \{R,B\}$ and a partition $\{Y_1, Y_2\}$ of $V(G)$ such that $E(Y_i, Y_i) \subseteq E^C(G)$ for all $i \in \{1,2\}$, and there are no disjoint $C$-colored edges in $E(Y_1, Y_2)$.
\end{lemma}

A graph $G$ is \emph{pancyclic} if it contains cycles of every length between 3 and $|V(G)|$. The \emph{girth} of a graph $g(G$) is the length of its shortest cycle, the \emph{circumference} of a graph $c(G)$ is the length of its longest cycle. A graph $G$ is \emph{weakly pancyclic} if it contains cycles of every length between $g(G)$ and $c(G)$. We shall make use of various theorems that assure that, under certain conditions, a graph is pancyclic or weakly pancyclic. The following lemma has a trivial proof.

\begin{lemma} \label{lemma:pancyclicalmostbipartite}
Let $k \ge 1$ and $G$ be a graph on $2k+1$ vertices. Let $\{V_1, V_2\}$ be a partition of $V(G)$ such that $|V_1| = k+1$ and $|V_2| = k$ and $E(V_1, V_2)$ form the edges of a complete bipartite subgraph. If there is an edge $e = \{x,y\} \subseteq V_1$, then $G$ is pancyclic.
\end{lemma}

%\begin{proof}
%We easily find the cycles of even length in the complete bipartite subgraph, and the cycles of odd length can be found using the edge $e \subseteq V_1$.
%Let $3 \leq j \leq 2k+1 = |G|$. We need to show that $C_j \subseteq G$. If $j = 2l$ is an even number, we can clearly find a cycle using $l$ vertices in each $V_i$ for $i \in \{1,2\}$. If $j = 2l+1$ is an odd number, then $1 \leq l \leq k$. Choose an edge $e = xy$ in $V_1$; select $l$ distinct vertices in $V_2$, and $l-1$ distinct vertices in $V_1 \setminus \{x,y\}$. We easily find a path of length $2j$ between $x$ and $y$ using the selected vertices, and thus, a $(2k+1)$-cycle in $G$.
%\end{proof}

\begin{theorem}[Bondy, \cite{bondy}] \label{lemma:bondy}
Let $G$ be a graph on $n \ge 3$ vertices with minimum degree at least $n/2$. Then $G$ is pancyclic, or $n = 2k$ and $G \cong K_{k,k}$.
\end{theorem}

\begin{corollary}[Dirac, \cite{dirac}] \label{corollary:bondy2}
Let $G$ be a graph on $n \ge 3$ vertices with minimum degree greater than $n/2$. Then $G$ is pancyclic.
\end{corollary}

\begin{theorem}[Brandt, \cite{brandt}] \label{lemma:brandt}
Let $G$ be a non bipartite graph on $n$ vertices with more than $(n-1)^2/4 + 1$ edges. Then $G$ is weakly pancyclic and contains a triangle.
\end{theorem}

More theorems and results about weakly pancyclic graphs will be stated in \Cref{section:proofofsecondarytheorem}. The next simple lemma ensures a bound on the girth of a graph given a lower bound on the minimum degree.

\begin{lemma} \label{lemma:boundongirth}
Let $G$ be a graph on $n \ge 9$ vertices with $\delta(G) \ge \frac{n}{4}$. Then $g(G) \leq 5$.
\end{lemma}

\begin{proof}
Let $C$ be a shortest cycle in $G$. Suppose $|C| \ge 6$. By the choice of $C$, each vertex in $V(C)$ has exactly two neighbors in $V(C)$. So, $|N(x) \setminus V(C)| \ge \delta(G) - 2$ for each $x \in V(C)$. Furthermore, for each pair of distinct vertices $x,y \in V(C)$, the sets $N(x) \setminus V(C)$ and $N(y) \setminus V(C)$ are disjoint, for the same reason. Thus, \[ n \ge |C| + \sum_{x \in V(C)} |N(x) \setminus V(C)| \ge |C| + |C| (\delta(G) - 2) = |C|(\delta(G) - 1), \]implying that $n/6 + 1 \ge \lceil n/4 \rceil$, a contradiction for $n \ge 9$.
\end{proof}

\begin{theorem}[Dirac, \cite{dirac}] \label{lemma:dirac}
Let $G$ be a 2-connected graph in $n \ge 3$ vertices. Then $c(G) \ge \min \{ 2 \delta(G), n \}$.
\end{theorem}

\section{Proof of \Cref{theorem:mainresult}}

In this section we prove \Cref{theorem:mainresult}. We assume \Cref{theorem:secondaryresult}, postponing its proof to \Cref{section:proofofsecondarytheorem}.

Let $k \ge 6$, so that $n = 2k+1 \ge 13$. Let $N \ge 5k+3$ and $G := K_{2k+1}$. Suppose there exists a red-blue edge-coloring of $K_N$ in a way such that $C_{2k+1}$ is not a red subgraph of $G$ and $W_{2k+1}$ is not a blue subgraph of $G$. By \Cref{lemma:surahmat}, we may asumme that $N \leq 6k$.

\Cref{theorem:secondaryresult} implies that $r(C_{2k+1},W_{2k+2}) \leq 5k+3$ when $k \ge 5$. Hence, $G$ contains a blue copy of $W_{2k+2}$ as a subgraph. Choose such a copy, and let $C$ be the ``rim'' of the wheel (the $(2k+2)$-cycle) and let $w$ be the ``hub'' of the wheel (the vertex not in $C$).

Consider the graph $G[C]$ with the induced edge-coloring of $G$. This graph does not contain a red copy of $C_{2k+1}$, as it would be present in $G$ as well. It also does not contain a blue copy of $C_{2k+1}$, as otherwise, adding $w$, we would create a blue copy of $W_{2k+1}$ in the graph $G$. So, the graph $G^R[C]$ (the graph induced by vertices of $C$, but only considering the red-colored edges) satisfies the hypotheses of  \Cref{lemma:partitionofhamiltoniangraph}. Hence, there exists a partition $\{U_1', U_2'\}$ of the vertices of $C$ such that $|U_1'| = |U_2'| = k+1$, both $U_1'$ and $U_2'$ induce complete red subgraphs, and there exists a vertex $v \in V(C)$ such that $E(U_1' - v, U_2' - v) \subseteq E^B(G)$. Without loss of generality, we suppose that $v \in U_2'$. We define $U_i = U_i' - \{v\}$ for each $i \in \{1,2\}$ and $U_3 = \{w\}$, as defined in the previous paragraph. So $|U_1| = k+1$ and $|U_2| = k$.

Recall that every edge contained in $U_i$ is red, for $i \in \{1,2,3\}$ and every edge between different pairs in $\{U_1, U_2, U_3\}$ is blue. We choose a triple $(X_1, X_2, X_3)$ of pairwise disjoint sets such that $U_i \subseteq X_i$ and every edge in $E(X_i, X_j)$ is blue; for each distinct $i, j \in \{1,2,3\}$. Assume $(X_1, X_2, X_3)$ maximizes the sum $|X_1| + |X_2| + |X_3|$ among all possible $3$-tuples satisfying the previous conditions. With this choice of $(X_1, X_2, X_3)$, we get the following lemma.

\begin{lemma} \label{lemma:xiastinducescompleteredgraphs}
Every edge contained in one of the sets $X_1$, $X_2$ or $X_3$ is red.
\end{lemma}

\begin{proof}
Let $i \in \{1,2\}$ and suppose that $X_i$ contains a blue edge $e$. As $|U_i| \ge k$ and $U_i$ is a complete red subgraph, this means that $|X_i| \ge k+1$. Also, $|X_{3-i}| \ge k$. So, by Lemma \ref{lemma:pancyclicalmostbipartite}, there is a blue monochromatic $(2k+1)$-cycle in $X_1 \cup X_2$. As $|X_3| \ge 1$ and every vertex in $X_3$ is joined by a blue edge to each vertex in $X_1 \cup X_2$, we get a blue copy of $W_{2k+1}$ in $G$, a contradiction.

Now, suppose that $X_3$ contains a blue edge $e = x_1 x_2$. As $C_{2k+1} \subseteq K_{k,k,1}$, we have a blue copy of $C_{2k+1}$ contained in $X_1 \cup X_2 \cup \{x_2\}$. But $x_1$ is joined with a blue edge to every vertex in $X_1 \cup X_2 \cup \{x_2\}$, so we find a blue copy of $W_{2k+1}$ as a subgraph in $G$, which is a contradiction. This proves the lemma.
\end{proof}

\Cref{lemma:xiastinducescompleteredgraphs} together with our assumption that $G$ contains no red copy of $C_{2k+1}$ imply that $|X_i| \leq 2k$ for each $i \in \{1,2,3\}$. If $\bigcup_{i=1}^3 X_i = V(G)$, then we are done. So, from now on, we suppose that $V(G) \setminus \bigcup_{i=1}^3 X_i \neq \emptyset$.

Let $v \notin \bigcup_{i=1}^3 X_i$. We show that there exists an $i \in \{1,2,3\}$ such that $E(v, X_i) \subseteq E^R(G)$. If it were not the case, there exist vertices $x_i \in X_i$ such that $v x_i$ is blue, for every $i \in \{1,2,3\}$. As $E(X_1, X_2) \subseteq E^B(G)$ and $|X_1|, |X_2| \ge k$, we can find a blue $(x_1, x_2)$-path in $X_1 \cup X_2$ of length $2k-1$. Along with the edges $v x_1$ and $v x_2$, we get a blue copy of $C_{2k+1}$ as a subgraph of $G$. But $x_3$ is joined with a blue edge to every vertex in $X_1 \cup X_2 \cup \{v\}$, so we find a blue copy of $W_{2k+1}$ as a subgraph in $G$, which is a contradiction.

This allows us to define \begin{align*}
W_1 & := X_1 \cup \{ v \in V(G) : E(v, X_1) \subseteq E^R(G) \}, \\
W_2 & := X_2 \cup \{ v \in V(G) : E(v, X_2) \subseteq E^R(G) \} \setminus W_1 \text{, and} \\
W_3 & := X_3 \cup \{ v \in V(G) : E(v, X_3) \subseteq E^R(G) \} \setminus (W_1 \cup W_2).
\end{align*} By the previous observations, $\{W_1, W_2, W_3\}$ is a partition of $V(G)$.

\begin{lemma} \label{lemma:rededgeswiwj}
Let $i \in \{1,2,3\}$ and $v \in W_i \setminus X_i$. Then $v$ has at least one red neighbor in $X_j$ for some $j \in \{1,2,3\} \setminus \{i\}$.
\end{lemma}

\begin{proof}
Immediate from the maximality of $|X_1| + |X_2| + |X_3|$.
\end{proof}

We say that a hedgehog in $G$ is \textit{red} if it is present in $G^R$, and \emph{blue} otherwise. Then each of the tuples $(W_i, X_i)$ for $i \in \{1,2,3\}$ are red hedgehogs. In particular, as $|X_1| \ge k+1$ and $|X_2| \ge k$, we have that both $(W_1, X_1)$ and $(W_2, X_2)$ are disjoint red hedgehogs that satisfy the hypothesis of \Cref{corollary:nodisjointrededges}, and so, if there are disjoint red edges in $E(W_1, W_2)$, there would be a red ($2k+1$)-cycle. This result can be strengthened, according to the following lemma.

\begin{lemma} \label{lemma:nodisjointw1w2redpaths}
There are no two vertex-disjoint red $(W_1, W_2)$-paths.
\end{lemma}

\begin{proof}
Suppose otherwise. Let $P_1, P_2$ be two vertex-disjoint $(W_1, W_2)$-paths of minimum joint length, and such that $\Vert P_1 \Vert \leq \Vert P_2 \Vert$. Let $p_1^1 \in W_1$ and $p_2^1 \in W_2$ be the endpoints of $P_1$; and $p_1^2 \in W_1$ and $p_2^2 \in W_2$ be the endpoints of $P_2$.

As $(W_1, X_1)$ is a red hedgehog and $|X_1| \ge k+1$, by \Cref{lemma:joininghedgehogbylongpaths} we can join the vertices $p_1^1$ and $p_1^2$ in $W_1$ with red paths of every length between $2$ and $k$. Similarly, we can join $p_2^1$ and $p_2^2$ in $W_2$ with red paths of every length between $2$ and $k-1$. Joining the paths $P_1$ and $P_2$ with the $(p_1^1, p_1^2)$-paths and $(p_2^1, p_2^2)$-paths previously mentioned, we obtain red cycles of every length between $\Vert P_1 \Vert + \Vert P_2 \Vert + 4$ and $\Vert P_1 \Vert + \Vert P_2 \Vert + 2k-1$. As $C_{2k+1}$ is not contained as a red subgraph in $G$, necessarily the bound \begin{align}
2k-2 \leq \Vert P_1 \Vert + \Vert P_2 \Vert \label{equation:p1p2equation1}
\end{align} holds. For each $i, j \in \{1,2\}$, we define $q_i^j$ as the vertex in $P_j$ adjacent to $p_i^j$ in the path $P_j$. Note that if $\Vert P_j \Vert \ge 3$, then the vertices $q_1^j$ and $q_2^j$ are distinct. As $(W_3, X_3)$ is a red hedgehog, there exists $x \in X_3$ such that $xv$ is a red edge for every other vertex $v \in W_3$.

%We show that if a path $P_j$ contains $x$ in its vertices, then $\Vert P_j \Vert \leq 4$.
Suppose the path $P_j$ contains $x$ in its vertices. As every vertex in $V(P_j) \subseteq W_3$ is a red neighbor of $x$, using $x$ we can find a strictly shorter path among the vertices of $P_j$, contradicting the minimality of $\Vert P_1 \Vert + \Vert P_2 \Vert$. We deduce that \begin{align}
\text{if a path $P_j$ contains $x$ in its vertices, then } \Vert P_j \Vert \leq 4. \label{equation:p1p2equation2}
\end{align} Now, suppose that $P_j$ does not contain the vertex $x$ and $\Vert P_{3-j} \Vert > 4$. Then, using (\ref{equation:p1p2equation2}) we deduce that $x$ does not belong in $ V(P_{3-j})$, and by hypothesis, $x \notin V(P_j)$. Then $p_1^{3-j} q_1^{3-j} x q_2^{3-j} p_2^{3-j}$ is a path of length 4 that is vertex-disjoint with $P_1$, which again contradicts the minimality of $\Vert P_1 \Vert + \Vert P_2 \Vert$. We have proved that %\begin{align} \forall j \in \{1,2\}, \ x \notin P_j \Rightarrow \Vert P_{3-j} \Vert \leq 4. \label{equation:p1p2equation3} \end{align}
\begin{align}
\text{if a path $P_j$ does not contain $x$ in its vertices, then } \Vert P_{3-j} \Vert \leq 4. \label{equation:p1p2equation3}
\end{align}

Using (\ref{equation:p1p2equation2}) and (\ref{equation:p1p2equation3}) together, the shorter of $P_1$ and $P_2$ has length at most $4$, and so it follows that $\Vert P_1 \Vert \leq 4$. Recalling the equation (\ref{equation:p1p2equation1}), we get that \begin{align}2k-6 \leq \Vert P_2 \Vert. \label{equation:p1p2equation5}
\end{align} As $2k+1 \ge 13$, we have that $\Vert P_2 \Vert > 4$. Therefore, using (\ref{equation:p1p2equation3}) we deduce that $x \in P_1$.

The edge $xy$ is red for every $y$ in $W_3$, and $x$ is contained in $P_1$, while the path $P_2$ is long; the idea is to use $x$ to construct shorter cycles using the vertices of $P_2$. Concretely, for every vertex $y \in V(P_2) \cap W_3$, we get that $p_1^1 P_1 x y P_2 p_1^2$ is a red $(p_1^1, p_1^2)$-path with no edge contained in $W_1$. We can choose $y \in V(P_2) \cap W_3$ in $\Vert P_2 \Vert - 1$ ways (every vertex in $P_2$, except its endpoints), hence, $xy P_2 p_1^2$ can be chosen of every length between $2$ and $\Vert P_2 \Vert$. Using (\ref{equation:p1p2equation5}), this implies the existence of red $(x, p_1^2)$-paths of every length between $2$ and $2k-6$.

Using these paths and the $(p_1^2, p_1^1)$-paths in $W_1$, we deduce the existence of red cycles in $G$ of every length between $6$ and $3k-5$. We conclude that $3k-5 < 2k+1$, which is false for $k \ge 6$. This contradiction proves the lemma.
%In $W_1$ we can find red $(p_1^2, p_1^1)$-paths of every length between $2$ and $k$, and joining these paths with the red $(x, p_1^2)$-paths previously mentioned, we obtain red cycles of every length between $4 + \Vert p_1^1 P_1 x \Vert$ and $3k - 6 + \Vert p_1^1 P_1 x \Vert$.
%Depending if $x = q_1^1$ or not, the path $p_1^1 P_1 x$ has length 1 or 2. Putting all together, we deduce the existence of red cycles of every length between $6$ and $3k-5$. We conclude that $3k-5 < 2k+1$, which is false for $k \ge 6$. This contradiction proves the lemma.
\end{proof}

Now we show that both $W_1$ and $W_2$ have no more than $2k$ vertices each.

\begin{lemma} \label{lemma:w1w2small}
$|W_i| \leq 2k$ for all $i \in \{1,2\}$.
\end{lemma}

\begin{proof}
Recall that $(W_1, X_1)$ is a red hedgehog. As $|X_1| \ge k+1$, we can choose $R_1, R_2 \subseteq X_1$ disjoint in a way such that $|R_1| = 1$ and $|R_2| = k$. Let $R_3 = W_1 \setminus (R_1 \cup R_2)$. Then these three subsets are disjoint and the edges between $E(R_i, R_j)$ are all red if $i \neq j$ in $\{1,2,3\}$. If $|W_1| > 2k$, then $|R_3| \ge k$ so we can easily find a red copy of $C_{2k+1}$, a contradiction.

If $|X_2| \ge k+1$ the conclusion follows from the same argument just presented, replacing $(W_1, X_1)$ with $(W_2, X_2)$. So it suffices to study the case where $|X_2| = k$ but $|W_2| \ge 2k+1$. By  \Cref{lemma:pancyclicalmostbipartite}, every edge contained in $W_2 \setminus X_2$ must be blue, and we get a complete blue subgraph of size at least $k+1$.

We claim that \begin{align}\text{each vertex outside $W_2$ has at most one red neighbor in $W_2 \setminus X_2$.} \label{equation:la5}
\end{align} Indeed, if this were not the case, we could find a red $(2k+1)$-cycle, using the complete bipartite red subgraph in $E(X_2, W_2 \setminus X_2)$ and an arbitrary red edge in $X_2$. Let $x_3$ be any vertex in $X_3$, and let $Y_2$ be the set of blue neighbors of $x_3$ in $W_2$. Because of \Cref{lemma:nodisjointw1w2redpaths}, there are no two disjoint red edges between $Y_2$ and $X_1$. By construction there are no red edges between $X_1$ and $X_2$, and by (\ref{equation:la5}) we conclude that no vertex in $X_1$ has two red neighbors in $Y_2$. Then the red edges in $E(Y_2, X_1)$ form a (possibly empty) star with its center in $Y_2$. Deleting this (possible) center of the star, we obtain a set $Y' \subseteq Y_2$ such that both $E(Y', x_3)$ and $E(Y', X_1)$ only contain blue edges. Furthermore, $E(X_1, X_2)$ and $E(x_3, X_1 \cup X_2)$ also contain only blue edges. We have that $|Y'| \ge 2$. Using at most two vertices of $Y'$ we find a blue cycle of length $2k+1$ in $X_1 \cup X_2 \cup Y'$, and joining $x_3$ we obtain a blue copy of $W_{2k+1}$, a contradiction.
\end{proof}

We now prove a lemma that will allow us to conclude \Cref{theorem:mainresult} afterwards.

\begin{lemma} \label{lemma:fuckinglemma}
Let $W, W'$ be disjoint sets of vertices in $V(G)$ such that \begin{enumerate}[label=\normalfont{\Alph*.}] \item there exists $X \subseteq W$ of size at least $k-1$ such that $(W,X)$ is a red hedgehog, \item $W'$ has size at least $3k+2$, \item there exists $X' \subseteq W'$ of size at least $k$ such that $X'$ induces a red clique, \item at least one of the following two hypothesis holds:
	\begin{enumerate}[label=\normalfont{D\arabic*.}]\item $W$ and $W'$ cover all vertices of $V(G)$ except at most one, and every edge between $W$ and $W'$ is blue, or \item $W$ and $W'$ cover all vertices of $V(G)$, and there exists a vertex in $W$ that only sends blue edges outside $W$. \end{enumerate}
\end{enumerate} Then there exists a partition $\{V_0, V_1, V_2, V_3\}$ of $V(G)$ such that $|V_0| \leq 2$, every edge inside the partition classes $\{ V_1, V_2, V_3 \}$ is red, and every edge between the partition classes $\{ V_1, V_2, V_3 \}$ is blue.
\end{lemma}

\begin{proof}
Let $H$ be the graph formed by the vertices in $W'$ with the induced edge coloring from $G$. By hypothesis B, $|V(H)| \ge 3k+2$. By \Cref{lemma:ramseyforcicles}, $r(C_{2k+2}, C_{2k+2}) = 3k+2$, so there is a monochromatic copy of $C_{2k+2}$ in $H$.

Note that both hypothesis D1 and D2 imply the existence of a vertex in $W$ that only sends blue edges to $H$. Now, $H$ does not contain a monochromatic $C_{2k+1}$, neither in red nor in blue (any vertex $v \in W$ that only sends blue edges to $H$ together with a blue $C_{2k+1}$ in $H$ would form a blue $W_{2k+1}$). So $H$ satisfies the hypotheses of \Cref{lemma:prenikiforovschelp}, and there exists a partition $\{Y_1, Y_2\}$ of $V(H)$ and a color $C \in \{R,B\}$ such that

\begin{enumerate} \item $Y_1$ and $Y_2$ both induce complete $C$-colored subgraphs, and \item $E(Y_1, Y_2)$ does not contain two disjoint $C$-colored edges. \end{enumerate}

Hence the edges in $E(H)$ that are not $C$-colored form a bipartite graph. By hypothesis C, there exists a red complete subgraph in $H$ of size at least $k \ge 3$, so the red edges in $H$ cannot form a bipartite subgraph. So, $C = R$. Using the fact that $|Y_1|, |Y_2| \leq 2k$ and hypothesis B, we deduce that $\min \{ |Y_1|, |Y_2| \} \ge k+2$. Define $Y_3 := W$. We have that $(Y_1, Y_1)$, $(Y_2, Y_2)$ and $(Y_3, X)$ are red hedgehogs satisfying the hypothesis of \Cref{corollary:nodisjointrededges}. So, there are no two disjoint red edges between each pair in $\{Y_1, Y_2, Y_3\}$, and so, the red edges induce a (possibly empty) red star between each of the pairs.

\textbf{Case A: Hypothesis D1 holds}: Then there can only be red edges between $Y_1$ and $Y_2$. So, choosing $V_0$ as the vertices not covered by $W \cup W'$ (at most one) together with the (possible) center of the red star between $Y_1$ and $Y_2$, we get that the edges between each pair in $\{Y_1 \setminus V_0, Y_2 \setminus V_0, Y_3 \setminus V_0 \}$ are all blue.

\textbf{Case B: Hypothesis D2 holds}: Deleting the centers of the three red stars between the pairs in $\{Y_1, Y_2, Y_3\}$ eliminates every red edge between these sets. We want to select at most two vertices in $V_0$, so it suffices to study the case where there is a red edge between each pair in $\{Y_1, Y_2, Y_3\}$.

Suppose first that we can find three red edges, one between each distinct pair of $\{ Y_1, Y_2, Y_3 \}$, such that the graph induced by these three edges is disconnected. Let $e_{ij}$ be the selected edges between $Y_i$ and $Y_j$, for each distinct $i, j \in \{1,2,3\}$. The red hedgehogs $Y_1, Y_2$ and $Y_3$ contain complete subgraphs of size at least $k+2$, $k+2$ and $k-1$, respectively. So, each edge $e_{ij}$ is adjacent to two hedgehogs with complete subgraphs of size at least $k-1$ and $k+2$. If the edges $\{e_{ij}\}_{i\neq j}^3$ induce a disconnected subgraph, we can choose an edge disjoint to the other two, and using \Cref{corollary:nodisjointrededges} we can join the endpoints of the disjoint edges in $\{e_{ij}\}_{i\neq j}^3$ with paths of every length between $2$ and $k-2$ or $k+1$, respectively. Using these paths we can find red cycles of every length between 9 and $2k+2$, in particular, a red copy of $C_{2k+1}$, a contradiction.

So, for every pair of three red edges $e_{ij} \in E(Y_i, Y_j)$ with $i < j \in \{1,2,3\}$; the graph induced by these three edges is connected. Every red edge between $\{ Y_1, Y_2, Y_3 \}$ is part of one of the three red stars, and so, contains at least one of the centers of these stars. If no edge contains the three centers, then we easily find three red edges between pairs in $\{Y_1, Y_2, Y_3\}$ inducing a disconnected subgraph, a contradiction. If one of these edges contain three centers of the stars, then choosing $V_0$ as the vertices of this edge, we get that every edge between each pair in $\{Y_1 \setminus V_0, Y_2\setminus V_0, Y_3\setminus V_0\}$ is blue, as required.

In every case: define $V_i := Y_i \setminus V_0$. As $Y_2$ and $Y_3$ are red complete graphs, it only remains to show that $Y_3$ is a complete red subgraph. If this were not the case, there is a blue edge $e = xy$ in $Y_3$. As $x$ only sends blue edges to $V_1 \cup V_2$, and each of $V_1$ and $V_2$ has size at least $k+1$, we find a blue $C_{2k+1}$ in $x \cup Y_1 \cup Y_2$. Adding $y$, we obtain a blue copy of $W_{2k+1}$, a contradiction.
\end{proof}

Recall that we have a partition of $V(G)$ in $\{W_1, W_2, W_3\}$ and there exists sets $X_i$ for each $i \in \{1,2,3\}$ such that $(W_i, X_i)$ are red hedgehogs for $i \in \{1,2,3\}$; we also have that $|X_1| \ge k+1$, $|X_2| \ge k$ and $|X_3| \ge 1$. By \Cref{lemma:w1w2small}, the sets $W = W_1$, $X = X_1$, $W' = V(G) \setminus W_1$ and $X' = X_2$ satisfy hypothesis A, B and C of \Cref{lemma:fuckinglemma}. The same holds if we replace the role of $W_1$ with $W_2$ and $X_1$ with $X_2$. So we can conclude \Cref{theorem:mainresult} immediately if hypothesis D2 of \Cref{lemma:fuckinglemma} is satisfied by $W_1$ or $W_2$. So, we may assume that \begin{align}
\text{every vertex in $W_i$ has a red neighbor outside $W_i$, for each $i \in \{1,2\}$.} \label{equation:rededgesfromwi}
\end{align}

By \Cref{lemma:nodisjointw1w2redpaths}, there are no two disjoint red $(W_1, W_2)$-paths. By Menger's Theorem applied to the graph $G^R$, we obtain that the size of a minimum $(W_1, W_2)$-separator in $G^R$ is at most one. Let $S \subseteq V(G)$ be such a separator. We separate the rest of the proof in two cases: there exists a red edge in $E(W_1, W_2)$ or not.

\textbf{Case A: There exists a red edge in $E(W_1, W_2)$}. Let $e = e_1 e_2$ be such an edge, with $e_i \in W_i$ for $i \in \{1,2\}$. As there are no two disjoint $(W_1, W_2)$-paths by \Cref{lemma:nodisjointw1w2redpaths}, the red edges in $E(W_1, W_2)$ form a non-empty star. Let $e_i$ be the center of this star. Every vertex in $W_i$, other than $e_i$, must send a red edge outside $W_i$ because of (\ref{equation:rededgesfromwi}), so it must be sent to $W_3$. The red edges in $W_3$ induce a connected graph, so if there were a red edge in $E(W_3, W_{3-i})$ disjoint from $e_{3-i}$, we could find two disjoint red $(W_1, W_2)$-paths, which is not possible. Using (\ref{equation:rededgesfromwi}) again, we see that every vertex in $W_{3-i}$ must send a red edge to $e_i \in W_i$. So the sets $W = W_{3-i}$, $X = X_{3-i}$, $W' = V(G) - W - e_i$ and $X' = X_i - e_i$ satisfy hypotheses A, B, C and D1 of \Cref{lemma:fuckinglemma}, and we are done.

%Let $e = e_1 e_2$ be such an edge, with $e_i \in W_i$ for each $i \in \{1,2\}$. As there are not two disjoint red $(W_1, W_2)$-paths, the red edges in $E(W_1, W_2)$ form a non-empty red star. Suppose that the center of this star, $z$, is in $W_i$. Then, any $u \in W_i \setminus z$ sends only blue edges to $W_{3-i}$. By (\ref{equation:rededgesfromwi}), every $u \in W_i \setminus z$ must send a red edge to $W_3$. Note that the red edges in $W_3$ form a connected graph, as $(W_3, X_3)$ is a red hedgehog with $X_3 \neq \emptyset$. Furthermore, if $u$ is a vertex in $W_{3-i}$ distinct from $e_{3-i}$, by (\ref{equation:rededgesfromwi}) it must have a red neighbor in $W_{i}$ or $W_3$, and it cannot be in $W_3$ because if it were, we could find two disjoint red $(W_1, W_2)$-paths. So $E(W_{3-i}, W_3)$ consists only of blue edges, and also every vertex in $W_{3-i}$ must have a red neighbor in $W_i$, which must be the center of the star, $z$. If we delete the vertex $z$ we delete every red edge in $E(W_{3-i}, W_i \setminus \{z\})$ and $E(W_{3-i}, W_3)$. So, $W_{3-i}$ and $z$ satisfy every hypothesis of \Cref{lemma:concluding3}, and we are done.

\textbf{Case B: Every edge in $E(W_1, W_2)$ is blue}. Remember that $S$ is a $(W_1, W_2)$-separator in $G^R$ of size at most one. It is not empty, because by (\ref{equation:rededgesfromwi}) every vertex in $W_1$ and $W_2$ has a red neighbor in $W_3$ and the red edges in $W_3$ form a connected subgraph, and this implies that there exists, at least, one red $(W_1, W_2)$-path. Furthermore, $S$ cannot be contained outside $W_3$, because deleting any vertex in $W_1 \cup W_2$ does not eliminate all the red $(W_1, W_2)$-paths. Therefore, we have $S = \{s\} \subseteq W_3$.

Suppose that $G^R[W_3] \setminus \{s\}$ is connected. Then there exists $i \in \{1,2\}$ such that $E(W_3 \setminus \{s\}, W_i )$ only contains blue edges, as otherwise there would still be red $(W_1, W_2)$-paths in $G^R - S$. Then the sets $W = W_i$, $X = X_i$, $W' = (W_3 \cup W_2) \setminus S$ and $X' = X_{3-i}$ satisfy hypotheses A, B, C and D1 of \Cref{lemma:fuckinglemma}, therefore concluding \Cref{theorem:mainresult}. So, we may suppose that $G^R[W_3] \setminus \{s\}$ is disconnected. Hence, $S = X_3$.

By \Cref{lemma:rededgeswiwj}, every vertex in $W_3 \setminus S$ has a red neighbor in $W_1$ or $W_2$. It cannot have red neighbors in both $W_1$ and $W_2$, as that would form a red $(W_1, W_2)$-path in $G-S$. So, \begin{align*}
Y_1 & := \{ v \in W_3 \setminus S : v \text{ has a red neighbor in } W_1 \} \text{ and} \\
Y_2 & := \{ v \in W_3 \setminus S : v \text{ has a red neighbor in } W_2 \},
\end{align*} together partition $W_3 \setminus S$. Furthermore, we have that every edge in $E(Y_1, W_2)$, $E(Y_2, W_1)$ and $E(Y_1, Y_2)$ is blue. So, we have a partition $\{ W_1 \cup Y_1, W_2 \cup Y_2, S \}$ of $V(G)$ such that every edge between $W_1 \cup Y_1$ and $W_2 \cup Y_2$ is blue. Suppose $|W_1 \cup Y_1| \leq |W_2 \cup Y_2|$ (otherwise, the proof is similar).

If $|W_1 \cup Y_1| \le 2k$, then $|W_2 \cup Y_2| \ge N - 1 - 2k \ge 3k+2$. Applying \Cref{lemma:fuckinglemma}, with $W = W_1 \cup Y_1$, $X = X_1$, $W' = W_2 \cup Y_2$ and $X' = X_2$; we conclude the theorem.

Suppose then that $|W_1 \cup Y_1| > 2k$. Since $W_1 \cup Y_1$ cannot contain a red $C_{2k+1}$, it contains a blue edge $e = v_1 v_2$. Consider $H := G[W_2 \cup Y_2]$. If there exists a vertex $w \in V(H)$ with at least $k$ blue neighbors in $H$, then by \Cref{lemma:pancyclicalmostbipartite}, we could find a blue $(2k+1)$-cycle which together with $w$ form a blue $(2k+1)$-wheel, which is impossible. Then, we have $\Delta^B(H) \leq k-1$, and $\delta^R(H) \ge |H|-k$.

We have that $|H| \ge |W_1 \cup Y_1| > 2k$. So, $\delta^R(H) > \frac{1}{2}|H|$ and by \Cref{corollary:bondy2}, $H^R$ is pancyclic, and thus contains a red copy of $C_{2k+1}$, a contradiction.

\section{Proof of \Cref{theorem:secondaryresult}} \label{section:proofofsecondarytheorem}

Both of the bounds of \Cref{theorem:secondaryresult} will follow from a more general type of bound.

\begin{definition}
Given two reals $\alpha \in [\frac{1}{4},1)$ and $\beta > 0$, we say that $(\alpha, \beta)$ is an \emph{admissible pair} if every 2-connected non-bipartite graph $G = (V,E)$ with $|V|=n$ and \[ \delta(G) \ge \alpha n + \beta \] contains every cycle $C_t$, for every $t$ such that $6 \leq t \leq c(G)$.
\end{definition}

Note that the non-bipartiteness of the graph is useful in the definition, because otherwise $K_{n/2,n/2}$ is a graph with $\delta(G) \ge n/2$ that is not weakly pancyclic. Given \Cref{lemma:boundongirth} and $\alpha \ge 1/4$, the condition of containing every cycle with length between $6$ and $c(G)$ is slightly weaker than being weakly pancyclic, for graphs with at least $9$ vertices, and it can be checked by inspection that the cycle condition is also satisfied by non-bipartite graphs with 8 vertices or less.

Brandt et al. \cite{brandt-faudree-goddard} proved some theorems concerning the values of $(\alpha, \beta)$ that assure weak pancyclicity of the graph, with or without the requirement of 2-connectedness.

\begin{theorem}[Brandt et al. \cite{brandt-faudree-goddard}] \label{lemma:brandt1}
Every non-bipartite graph of order $n$ with minimum degree $\delta(G) \ge \frac{n+2}{3}$ is weakly pancyclic with girth at most 4.
\end{theorem}

This implies that $(\frac{1}{3}, \frac{2}{3})$ is an admissible pair.

\begin{theorem}[Brandt et al. \cite{brandt-faudree-goddard}] \label{lemma:brandt2}
Every non-bipartite 2-connected graph of order $n$ with minimum degree $\delta(G) \ge \frac{n+1000}{4}$ is weakly pancyclic unless $G$ has odd girth $7$, in which case it has every cycle from 4 up to its circumference except the 5-cycle.
\end{theorem}

This implies that $(\frac{1}{4}, 250)$ is an admissible pair. Brandt et al. also give an example to show that no admissible pair $(\alpha, \beta)$ has $\alpha < 1/4$. Take two copies of $K_{m,m}$ intersecting in one vertex and join one vertex on the opposite side of the intersection vertex in one $K_{m,m}$ to such a vertex in the other $K_{m,m}$. Then this graph has $n:=4m-1$ vertices, minimum degree $(n+1)/4$, is 2-connected, Hamiltonian and has a triangle, but it is not weakly pancyclic as it does not contain any even cycle of length more than $(n+1)/2$.

We prove a bound on the Ramsey number $r(C_{2k+1}, W_{2j})$ for $k<j$, that depends on the existence of an admissible $(\alpha, \beta)$ pair.

\begin{theorem} \label{theorem:generalsecondaryresult}
Let $(\alpha, \beta)$ be an admissible pair, and $2 < k < j$ integers. Then \[ r(C_{2k+1}, W_{2j}) \leq \frac{3j + \beta}{1 - \alpha}. \] 
\end{theorem} Then, using $j = k+1$ with the admissible pairs implied from \Cref{lemma:brandt1} and \Cref{lemma:brandt2} respectively we obtain easily the bounds of \Cref{theorem:secondaryresult}.

Now let us prove \Cref{theorem:generalsecondaryresult}. Let $2 < k < j$ be integers and $(\alpha, \beta)$ an admissible pair with $\alpha \ge 1/4$. Let $G$ be a graph with $|G| \ge (3j+\beta)/(1 - \alpha)$ and $c: E(G) \rightarrow \{R,B\}$ a red-blue coloring of its edges. Suppose, for the sake of contradiction, that $C_{2k+1} \nsubseteq G^R$ or $W_{2j} \nsubseteq G^B$.

The bound $\alpha \ge 1/4$ implies that $|G| \ge 4j + \frac{4}{3} \beta > 4j$, and thus, $|G| \ge 4j+1$. We start with a series of lemmas.

\begin{lemma} \label{lemma:gr2conexo}
If $\delta^R(G) \ge j$, then $G^R$ is 2-connected.
\end{lemma}

\begin{proof}
Suppose that $\delta^R(G) \ge j$ but $G^R$ is not 2-connected. Then there exists a $S \subseteq V(G)$ with $|S| \leq 1$ such that $G^R - S$ is disconnected. Let $C_1, \dotsc, C_r$ be the connected components of $G^R - S$, such that $|C_1| \ge |C_2| \ge \dotsb \ge |C_r|$. We have that $\delta^R(G-S) \ge \delta^R(G) - |S| \ge j - 1$ and so, every connected component in $\{C_i\}_{i \in [r]}$ has size at least $j$.

%As $G^R - S$ is disconnected, we must have at least two connected components. Suppose we have at least three connected components in $G^R - S$. Each one of them has size at least $j$, and for each $i < j$, the set $E(C_i, C_j)$ has only blue edges. Then, selecting $j$ vertices in $C_1$, $j$ vertices in $C_2$ and $1$ vertex in $C_3$ we find a blue copy of $W_{2j}$, which is a contradiction. So, $G^R - S$ has exactly two connected components.
%Note that $|C_1| \ge 2j$. If there is a vertex $v$ in $C_1$ with at least $j$ blue neighbors in $C_1$, selecting $j$ vertices in $C_2$ we can find a blue $C_{2j}$ in $N^B(v)$ and so we find a blue $W_{2j}$, which is impossible. So $\Delta^B(C_1) < j$; and therefore $\delta^R(C_1) \ge |C_1| - j$. This implies that $\delta^R(C_1) \ge |C_1|/2$, and by \Cref{lemma:bondy} then $C_1^R$ is pancyclic or complete bipartite with parts of equal size.
Thus, as all edges between distinct $C_i$, $C_j$ are all blue and $G^B$ does not contain $W_{2j}$, we see that $G^R-S$ has exactly two connected components; with $|C_1| \ge 2j$. For the same reason, $\delta^R(C_1) \ge |C_1| - j \ge |C_1|/2$, and by \Cref{lemma:bondy} $C_1^R$ is pancyclic or complete bipartite with parts of equal size.

%If $\delta^R(C_1) > |C_1|/2$ then by \Cref{corollary:bondy2} we have that $C_1^R$ is pancyclic and so, it contains every cycle $C_t$ with $3 \leq t \leq |C_1|$. As $|C_1| \ge 2j \ge 2k+1$, it contains a red copy of $C_{2k+1}$, which is not possible. So we may suppose that $\delta^R(C_1) = |C_1|/2$ and that $C^R_1$ is a complete bipartite graph with parts of equal size. In this case, we can easily deduce that $|C_1| = |C_2| = 2j$.
If $C_1^R$ is pancyclic, then, as $|C_1| \ge 2j \ge 2k+1$, we find a red copy of $C_{2k+1}$, which is not possible. So $C_1^R$ is a complete bipartite graph. In particular, by \Cref{corollary:bondy2}, $\delta^R(C_1) = |C_1|/2$ and thus, $|C_1|=|C_2| = 2j$.

%So in $C_1$ we have a complete bipartite graph in the red edges; with each part, $C_{1,1}$ and $C_{1,2}$, of size exactly $j$. Each of $C_{1,1}$ and $C_{1,2}$ induces a complete blue subgraph of size $j$.
If there is a vertex $v$ in $C_2$ with two blue neighbors in $C_2$, then we find a blue $W_{2j}$ with $v$ as hub. So $\Delta^B(C_2) < 2$ and thus, $\delta^R(C_2) \ge |C_2| - 2 > j \ge |C_2|/2$. By \Cref{corollary:bondy2}, $C_2^R$ is pancyclic and therefore contains a red $C_{2k+1}$,a contradiction. %This proves the lemma.
\end{proof}

\begin{lemma} \label{lemma:grnobipartito}
$G^R$ is not bipartite.
\end{lemma}

\begin{proof}
Otherwise, as $|G| \ge 4j+1$, one of the parts of the bipartition has at least $2j+1$ vertices, therefore containing a blue copy of $W_{2j}$ as a subgraph, which is impossible.
\end{proof}

\begin{lemma} \label{lemma:highbluedegree}
$\Delta^B(G) \ge 3j$.
\end{lemma}

\begin{proof}
We prove the stronger claim that $\Delta^B(G) > (1-\alpha)|G|-(\beta+1)$. Suppose otherwise, that is, $\delta^R(G) \ge \alpha |G| + \beta$. This implies that $\delta^R(G) \ge |G|/4 \ge j$. Then, by \Cref{lemma:gr2conexo}, $G^R$ is 2-connected. By \Cref{lemma:grnobipartito}, $G^R$ is not bipartite. As $(\alpha, \beta)$ is an admissible pair, $G^R$ contains cycles of every length between $6$ and $c(G^R)$. As the $(2k+1)$-cycle is not a subgraph of $G^R$, this means that $c(G^R) < 2k+1$. By \Cref{lemma:dirac}, $\min \{ 2 \delta^R(G), |G| \} \le c(G^R)$ and this implies that $2 \delta^R(G) < 2k+1$. From this, it follows that $\delta^R(G) < k < j \leq \delta^R(G)$, a contradiction.
%So, $\delta^R(G) < \alpha |G| + \beta$. This means that $\Delta^B(G) \ge |G| - 1 - \delta^R(G) > (1 - \alpha)|G| - (\beta + 1)$ and the claim follows.
\end{proof}

From now on, let $w$ be a vertex of maximum blue degree and let us define $H := N^B(w)$. From \Cref{lemma:highbluedegree}, we get that $|H| \ge 3j$.

\begin{lemma} \label{lemma:hrcircunferencia}
$c(H^R) > 2k+1$.
\end{lemma}

\begin{proof}
As $|H| \ge 3j > r(C_{2j})$, $H$ contains a monochromatic copy of $C_{2j}$. It cannot be blue, as that would create, together with the vertex $w$, a blue copy of $W_{2j}$. Thus $H$ contains a red copy of $C_{2j}$ and this implies that $c(H^R) \ge 2j > 2k+1$.
\end{proof}

\begin{lemma} \label{lemma:hbnobiparito}
$H^B$ is not bipartite.
\end{lemma}

\begin{proof}
Suppose $H^B$ is bipartite with $H_1, H_2$ the parts of the bipartition. As $G$ does not contain red copies of $C_{2k+1}$, each part of the bipartition has size at most $2k \leq 2j - 2$. As $|H| \ge 3j$, we have that $|H_1|, |H_2| \ge j+2$.

%This implies that \begin{align*} \min \{ |H_1|, |H_2| \}
%& = |H| - \max \{ |H_1|, |H_2| \} \\
%& > (1- \alpha) |G| - (\beta + 1) - 2j + 2 \\
%& \ge j +1, \end{align*} and so, $\min \{ |H_1|, |H_2| \} \ge j+2$.

Note that $E(H_1, H_2)$ cannot have two disjoint red edges, as that would form red cycles of every size between $4$ and $|H|$, including $2k+1$. So the red edges in $E(H_1, H_2)$ form a (possibly empty) star. Let $x \in H$ be the center of that star. Then $E(H_1 - x, H_2 - x)$ contains a blue copy of $C_{2j}$ as a subgraph. This copy, joined with the vertex $w$ forms a blue copy of $W_{2j}$, a contradiction.
\end{proof}

\begin{lemma} \label{lemma:hrnobiparito}
$H^R$ is not bipartite.
\end{lemma}

\begin{proof}
Suppose $H^R$ is bipartite with $H_1$ and $H_2$ the parts of the bipartition. As $G$ does not contain blue copies of $C_{2j}$, each part of the bipartition has size less than $2j$. Without loss of generality, suppose that $|H_1| \leq |H_2|$. As $|H| \ge 3j$, this implies that $|H_2| \ge j+1$.

Note that $E(H_1, H_2)$ cannot have two disjoint blue edges, as that would form blue cycles of every size between $4$ and $|H|$, including $2j$. So the blue edges in $E(H_1, H_2)$ form a (possibly empty) star. Let $x$ be the center of this star, so $E(H_1 - x, H_2 - x)$ only contains red edges.

We claim that we can choose $x$ such that $H_2 - x$ has at least $j+1$ vertices. If $|H_2| \ge j+2$ then this is obvious. If $|H_2| = j+1$, using that $|H| \ge 3j$ we deduce that $|H_1| = 2j-1$. If there is a vertex in $H_2$ with two blue neighbors in $H_1$, they would form a blue copy of $C_{2j}$, which is impossible. So, without loss of generality, in this case we can consider that $x$, the center of the blue star in $E(H_1, H_2)$, is in $H_1$. So, in every case, $|H_2 - x| \ge j+1$.

Let $Z = N^R(w)$, the red neighbors of $w$ (recall that $H = N^B(w)$). We now show that $E(H_1-x,Z)$ or $E(H_2-x,Z)$ only contains blue edges. Otherwise, there exist two red edges $x_1 z_1$ and $x_2 z_2$ with $x_1 \in H_1 - x$ and $x_2 \in H_2 - x$ such that $z_1, z_2 \in Z$. If $z_1 = z_2$, then $x_1 z_1 x_2$ is a red path of length 2 outside $H$. As $E(H_1 - x, H_2 - x)$ only contains red edges, we easily find a red copy of $C_{2k+1}$ using the previously mentioned path, which is impossible. If $z_1 \neq z_2$, then $x_1 z_1 w z_2 x_2$ is a red $(x_1,x_2)$-path of size 4 and we can find a red $(2k+1)$-cycle similarly as before. Let $i \in \{1,2\}$ such that $E(H_i-x,Z)$ only contains blue edges.

We have that $|H_i - x| \ge j+1$. Notice that every vertex $y \in H_i - x$ has every other vertex in $H_i - x$ as blue neighbors, as well as every vertex in $Z$ and $\{w\}$. So, $( H_i - x, H_i - x \cup \{w\} \cup Z )$ is a blue hedgehog with $|H_i - x| \ge j+1$ and $|H_i - x \cup \{w\} \cup Z| \ge |G| - |H_{3-i}| - 1 \ge 2j+1$. So, selecting an arbitrary vertex $y \in H_i - x$ as the hub, we can find a blue copy of $W_{2j}$, a contradiction.
\end{proof}

\begin{lemma} \label{lemma:hbweaklypancyclic}
$H^B$ is weakly pancyclic and contains a triangle.
\end{lemma}

\begin{proof}
Note that there are more than $(|H|-1)^2/4 + 1$ red edges in $H$, or more than $(|H|-1)^2/4 + 1$ blue edges in $H$. By \Cref{lemma:hbnobiparito} and \Cref{lemma:hrnobiparito}, neither $H^R$ nor $H^B$ are bipartite graphs. This, along with \Cref{lemma:brandt} implies that one of the graphs $H^R$ or $H^B$ is weakly pancyclic and contains a triangle. If $H^R$ is weakly pancyclic and contains a triangle, \Cref{lemma:hrcircunferencia} implies that $H^R$ contains a $(2k+1)$-cycle, a contradiction. This proves the lemma.
\end{proof}

\begin{lemma} \label{lemma:gradominimohb}
$\delta^B(H) \ge j$.
\end{lemma}

\begin{proof}
Suppose not. Then there is a vertex $v \in H$ with less than $j$ neighbors in $H$. As $|H| \ge 3j$, we have that $|H - v| \ge 3j-1 = r(C_{2j})$. As $H$ cannot contain a blue copy of $C_{2j}$, this means that we have a red copy of $C_{2j}$ in $H$. Let $C$ be such a cycle and $V(C) = \{ v_0, v_1, \dotsc, v_{2j-1} \}$ be the set of its vertices . The vertex $v$ has less than $j$ blue neighbors in $H$, in particular, more than half of the vertices in $V(C)$ are red neighbors of $v$. Thus, by the pigeonhole principle, there are two neighbors $x, y$ of $v$ $V(C)$ at distance $2k-1$ in $C$. Joining the red path between $x$ and $y$ with $v$ we obtain a red copy of $C_{2k+1}$, a contradiction.
%Let us show that there exists two red neighbors of $v$ in $V(C)$ at distance $2k-1$ in $C$. If it were not, then the function $v_i \mapsto v_{(i+(2k-1) \mod 2j)}$ would be an injective function from the set of red neighbors of $v$ in $V(C)$ to the set of blue neighbors of $v$ in $V(C)$, which is evidently absurd.
%Then, there exists $i \in \{0, \dotsc, 2j-1 \}$ such that $v_i$ and $v_{(i+(2k-1) \mod 2j)}$ are joined by a red path in $V(C)$ of length $2k+1$. Together with the red edges joining $v_i$ and $v_{(i+(2k-1) \mod 2j)}$ with $v$, we form a red copy of $C_{2k+1}$, a contradiction.
\end{proof}

\begin{lemma} \label{lemma:hbno2conexo}
$H^B$ is not $2$-connected.
\end{lemma}

\begin{proof}
If $H^B$ were $2$-connected, then by \Cref{lemma:dirac}, and by \Cref{lemma:gradominimohb}, $c(H^B) \ge 2j$. So, by \Cref{lemma:hbweaklypancyclic}, $H^B$ contains a $2j$-cycle, that together with the vertex $w$ becomes a blue copy of $W_{2j}$, a contradiction.
\end{proof}

We are now ready to prove \Cref{theorem:secondaryresult}. 

\begin{proof}[Proof of \Cref{theorem:secondaryresult}]
As $H^B$ is not $2$-connected by \Cref{lemma:hbno2conexo}, there exists a $S \subseteq H$ with $|S| \leq 1$ such that $H^B - S$ is disconnected. By \Cref{lemma:gradominimohb} we have that $\delta^B(H-S) \ge k$ and so, every connected component in $H^B - S$ has size at least $k+1$. We must have exactly two connected components in $H^B - S$; if we had more than two then we could find a red copy of $C_{2k+1}$ in $H$. Let $C_1, C_2$ be the sets of vertices of the connected components in $H^B - S$. Then $\{C_1, C_2, S\}$ form a partition of $H$ (with $S$ possibly empty).

By \Cref{lemma:pancyclicalmostbipartite} we can easily see that $C_1$ and $C_2$ induce complete blue subgraphs. If $S$ is empty, then $V(H) = C_1 \cup C_2$ and so, $H^R$ would be a bipartite graph, in contradiction with \Cref{lemma:hrnobiparito}. So, the vertex $s \in S$ must have red neighbors both in $C_1$ and $C_2$, and we find a red copy of $C_{2k+1}$ in $H$, giving the final contradiction.
\end{proof}

\begin{footnotesize}
\bibliographystyle{alpha}
\bibliography{biblo}
\end{footnotesize}

\end{document}